%% file: constraint_minors170824.tex
\documentclass[11pt]{article}

\input{minimalJ}

\usepackage{verbatim}
\usepackage{enumerate}
\usepackage[all]{xy}

\usepackage{subfig}

\title{Embedding simply connected \\ 2-complexes in 3-space\\ \Large III. Constraint minors}

\author{Johannes Carmesin
\medskip 
\\
  {University of Cambridge}
}

\newcommand{\sm}{\setminus}

\begin{document}

\maketitle

\begin{abstract}
We characterise the following property by six obstructions: given a graphic matroid 
$M$ and a set $X$ of its elements, when is $M$ the cycle matroid of a graph $G$ such 
that $X$ is a connected edge set in $G$? 
\end{abstract}

\section{Introduction}
For a purely graph-theoretic introduction read \autoref{sec:graph}. 

Tutte \cite{TutteBook} proved that a matroid can be represented by a graph if and only if 
it has no minor isomorphic to $U_{2,4}$, the fano-plane, the dual fano-plane or the dual matroids 
of the two nonplanar graphs $K_5$ or $K_{3,3}$. The topic of this paper is the following related 
reconstruction question: given a graphic matroid 
$M$ and a set $X$ of its elements, when is $M$ the cycle matroid of a graph $G$ such 
that $X$ is a connected edge set in $G$? Our motivation for studying that question is that in 
\cite{3space4} it  arises when characterising embeddability in 3-space of 
certain 2-complexes by excluded minors. 
\vspace{.3cm}

A \emph{constraint matroid} is a pair $(M,X)$, where $M$ is a matroid and $X$ is a set of 
elements  
of $M$.
A constraint matroid $(M,X)$ is \emph{realisable} if $M$ is the cycle matroid of a graph 
$G$ such that $X$ is a connected edge set in $G$. The class  
of constraint matroids $(M,X)$ that are realisable is closed 
under contracting arbitrary elements and deleting elements not in $X$. 
A constraint matroid obtained by these operations from $(M,X)$ is a \emph{constraint minor} of 
$(M,X)$. In this paper we characterise the class of the realisable  (graphic) constraint
matroids by excluded constraint minors.

\begin{thm}\label{new_intro}
A graphic constraint matroid is realisable if and only if it does not have one 
of the six constraint minors depicted in \autoref{fig:k4}, 
\autoref{fig:wheel} or 
\autoref{fig:constraint_prisms}.
\end{thm}
All these six obstructions are 3-connected and graphic. So we 
just depict their unique graphs. \autoref{new_intro} can be restated in purely graph theoretic 
terms, see \autoref{main_intro} below.
   \begin{figure} [htpb]   
\begin{center}
   	  \includegraphics[height=1.5cm]{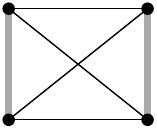}
   	  \caption{The constraint $K_4$. The edges in $X$ are depicted grey. }\label{fig:k4}
\end{center}
   \end{figure}

   \begin{figure} [htpb]   
\begin{center}
   	  \includegraphics[height=1.5cm]{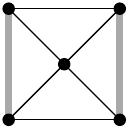}
   	  \caption{The constraint wheel. The edges in $X$ are depicted grey. }\label{fig:wheel}
\end{center}
   \end{figure}
   \begin{figure} [htpb]   
\begin{center}
   	  \includegraphics[height=2cm]{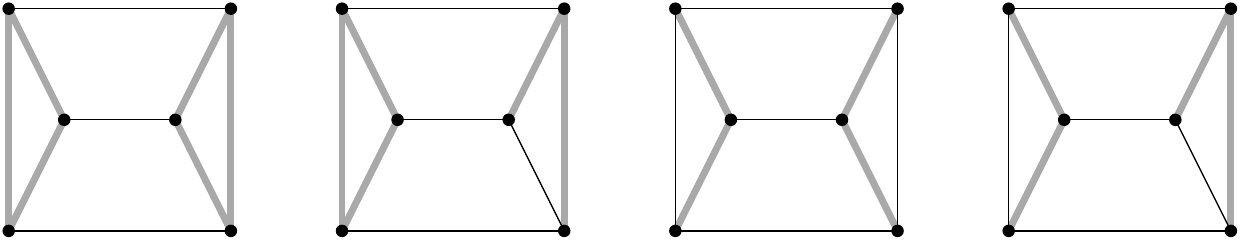}
   	  \caption{The four constraint prisms. The edges in $X$ are depicted grey. 
}\label{fig:constraint_prisms}
\end{center}
   \end{figure}

  \section{A graph theoretic perspective}\label{sec:graph}
  
   Although \autoref{new_intro} is about matroids, most of this paper is about the following 
equivalent graph theoretic version.

A \emph{constraint graph} is a pair $(G,X)$, where $G$ is a graph and $X$ is an edge set of $G$.
A constraint graph is \emph{constraint connected} if $X$ is a connected edge set in $G$. The class  
of constraint graphs $(G,X)$ that are constraint connected is closed 
under contracting arbitrary edges and deleting edges not in $X$. 
A constraint graph obtained by these operations from $(G,X)$ is a \emph{constraint minor} of 
$(G,X)$. 
It is straightforward to show that a 2-connected\footnote{A constraint graph 
$(G,X)$ is \emph{$k$-connected} if $G$ is $k$-connected.} constraint graph $(G,X)$ is constraint 
connected if and only if it has no constraint minor isomorphic to the 4-cycle whose constraint 
consists of two opposite edges. The analogue question for connected graphs is not much more 
interesting. 

However, it turns out that the question gets nontrivial if we restrict our attention to 3-connected 
graphs. 
 \begin{thm}\label{main_intro}
A 3-connected constraint graph $(G,X)$ is constraint connected if and only if it does not have one 
of the six (3-connected) constraint minors  depicted in \autoref{fig:k4}, \autoref{fig:wheel} or 
\autoref{fig:constraint_prisms}.
\end{thm}

It is straightforward to deduce \autoref{main_intro} from \autoref{new_intro} above. However the 
converse is also true as follows.

\begin{proof}[Proof that \autoref{main_intro} implies \autoref{new_intro}.]
Let $(M,X)$ be a constraint matroid. If $M$ is 3-connected, then it is the cycle matroid of a 
unique 
graph $G$ by a theorem of Whitney \cite{Whitney_flip}. In this case \autoref{new_intro} for $(M,X)$ 
is a 
restatement of \autoref{main_intro} for $(G,X)$. 
 
Now let $(M,X)$ be a constraint matroid that has no constraint minor depicted in  \autoref{fig:k4}, 
\autoref{fig:wheel} or 
\autoref{fig:constraint_prisms}. It remains to show that $(M,X)$ is realisable. 
Since a constraint matroid is realisable if and only if each of its 2-connected components is, 
we may assume that $M$ is 2-connected. 

Now we prove by induction that $(M,X)$ is realisable. The base case is that $M$ is 3-connected.

If $M$ is not 3-connected, its Tutte-decomposition \cite{TutteGrTh} has a non-trivial 
2-separation $(A,B)$. 
Let $M_1$ and $M_2$ be the two matroids obtained by decomposing $M$ along the 2-separation 
$(A,B)$. In particular, $M_1$ and $M_2$ both contain a virtual element $e$ and the 
2-sum\footnote{See 
\cite{oxley2} for a definition. } of $M_1$ and $M_2$ along $e$ is $M$.
Note that the $M_i$ can be obtained from $M$ by contracting elements and replacing a parallel 
class by the virtual element $e$. 
For $i=1,2$, let $(M_i,X_i)$ be the constraint matroid, where $X_i$ is $X\cap E(M_i)$ plus possibly 
$e$ if $M_{i+1}$ 
contains a circuit $o$ such that $o-e\se X$. 
It is straightforward to check that the $(M_i,X_i)$ are 
constraint minors of $M$. Hence by induction, they are realisable. Let $G_i$ 
be a graph realising 
$(M_i,X_i)$.

Let $G$ be the 2-sum of the graphs $G_1$ and $G_2$ along the 
virtual element $e$. By construction $M$ is the cycle matroid of $G$. If the virtual element $e$ is 
in 
$X_1$ or $X_2$, it is straightforward to see that $(G,X)$ is constraint connected. So $M$ is 
realisable. So we may assume that $e$ is in no $X_i$. If one of the $X_i$ is 
empty, then $(G,X)$ is constraint connected. So we may assume that both $X_i$ are nonempty.

Then not only $(M_1,X_1)$ but also $(M_1,X_1+e)$ is a constraint minor of $(M,X)$. 
So by induction there is a graph $G_1'$ realising $(M_1,X_1+e)$. In $G_1'$ an element 
of the set $X_1$ is incident with an endvertex of $e$. Similarly, there is a graph $G_2'$ realising 
$(M_2,X_2+e)$, and there is an element of the set $X_2$ 
is incident with an endvertex of $e$.
Let $G'$ be the 2-sum of the graphs $G_1'$ and $G_2'$. By flipping\footnote{By a theorem 
of Whitney, graphs represented by a 2-connected matroid are 
unique up to 
flipping 2-separators \cite{Whitney_flip}. } the 2-separator given by the endvertices of $e$ in 
$G'$ 
if necessary, 
we ensure that $X$ is connected in $G'$. Put another way, $(G',X)$ is constraint connected 
witnessing that $(M,X)$ is realisable.
 \end{proof}

Hence the rest of this paper is dedicated to the proof of \autoref{main_intro}, which is purely 
graph-theoretic. 
Before jumping into the proof, let us fix a few lines of notation.
In this paper all graphs are simple. 
In particular, if we contract an edge, we afterwards delete all but one edge from every parallel 
class. In the context of a constraint graph $(G,X)$, we first delete edges in a parallel 
classes that are not in $X$ (so that constraint minors on simple graphs preserve constraint 
connectedness). Throughout this paper we follow the convention that the empty set is a connected 
edge set in $G$. Beyond that we follow the notation of \cite{DiestelBookCurrent}. 
Let's get started 
with the proof.

\section{Deleting and contracting edges outside the constraint}

In this section we prove \autoref{summary} below, which is used in the proof of 
\autoref{main_intro}.

Given a constraint graph $(G,X)$, an edge $e$ not in $X$ is \emph{essential} if neither $(G/e,X)$ 
nor $(G\sm e,X)$ has a 3-connected constraint minor $(G',X')$ such that $X'$ is disconnected.  
Informally, \autoref{summary} below gives a structural description of the constraint graphs $(G,X)$ 
in which every edge not in $X$ is essential. 

Before we can prove \autoref{summary} we need some 
preparation. 
Our first aim is to prove the following.

\begin{lem}\label{situation3}
Let $(G,X)$ be a 3-connected constraint graph that is not constraint connected.
Assume that every edge not in $X$ is essential.
Then $G[X]$ has precisely two connected components or $(G,X)$ is the weird prism (defined in 
\autoref{bad_prims}).
\end{lem}
First we consider some particular examples that will come up in the proof of \autoref{situation3}.

\begin{eg}\label{bad_prims}
The \emph{weird prism} is the pair $(P,X)$, where $P$ is the prism and $X$ consists of the three 
edges in the complement of the two triangles, see \autoref{fig:weird}.
Contracting any particular edge in $X$, gives the constraint wheel. 
   \begin{figure} [htpb]   
\begin{center}
   	  \includegraphics[height=2.3cm]{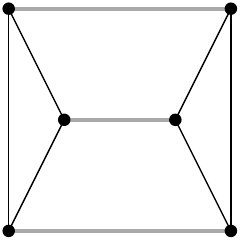}
   	  \caption{The weird prism. The edges in $X$ are depicted grey. 
}\label{fig:weird}
\end{center}
   \end{figure}
\end{eg}

\begin{eg}\label{wagner}
The \emph{constraint Wagner graph} is the pair $(W,X)$, where $W$ is the Wagner graph and $X$ is 
the set 
of edges in the complement of one of its six-cycles, see \autoref{fig:wagner}. 
If we contract a single edge of $X$, we get the constraint wheel. If we contract any two opposite 
edges on the six cycle, then we get a constraint $K_4$. 
   \begin{figure} [htpb]   
\begin{center}
   	  \includegraphics[height=2.3cm]{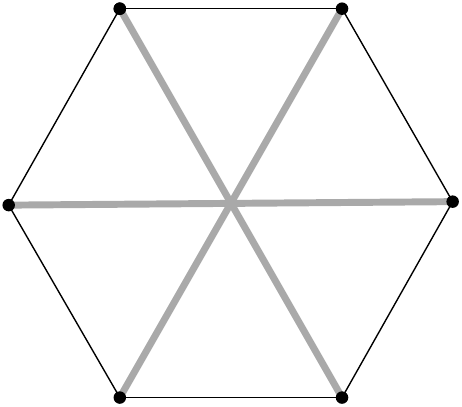}
   	  \caption{The constraint Wagner graph. The edges in $X$ are depicted grey. 
}\label{fig:wagner}
\end{center}
   \end{figure}
\end{eg}

\begin{eg}\label{wagner_prism}
The \emph{Wagner prism} is the pair $(W',X')$, where $W'$ is the prism and $X'$ contains one 
edge not in the two triangles of the prism. The two other edges in $X'$ are the only two edges of 
the prism in the triangles that are vertex-disjoint to that edges, see \autoref{fig:wagner_prism}. 
There are two opposite edges on the six cycle formed by the edges not in $X'$ whose contraction 
gives the constraint $K_4$. 
   \begin{figure} [htpb]   
\begin{center}
   	  \includegraphics[height=2.3cm]{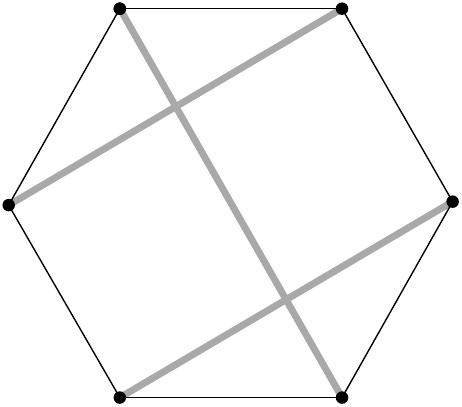}
   	  \caption{The Wagner prism. The edges in $X$ are depicted grey. 
}\label{fig:wagner_prism}
\end{center}
   \end{figure}
\end{eg}

\begin{lem}\label{contract_or}
Let $(G,X)$ be a 3-connected constraint graph such that $G[X]$ has at 
least 3-connected components.
Assume that $G$ is not the constraint Wagner graph, not the weird prism and not the Wagner prism.
Then there is a 3-connected constraint minor $(G',X')$ of $(G,X)$ such that $X'$ is disconnected in 
$G'$ and such that $E(G')\sm X'$ is a proper subset of $E(G)\sm X$.
\end{lem}

\begin{proof}[Proof that \autoref{contract_or} implies \autoref{situation3}.]
By \autoref{wagner}, the constraint Wagner graph has an edge not in $X$ that is not essential.
Thus $(G,X)$ is not the constraint Wagner graph. Similarly, $(G,X)$ is not the Wagner prism by 
\autoref{wagner_prism}. Hence by \autoref{contract_or}, $G[X]$ has 
precisely two 
connected components  or is the weird prism.
\end{proof}

\begin{proof}[Proof of \autoref{contract_or}.]
Let $e$ be an arbitrary edge not in $X$.
 If the simple graph $G'=G/e$ is 3-connected, then $(G', X\cap E(G'))$ is the desired constraint 
minor.
Otherwise by Bixby's Lemma \cite{oxley2} the graph $G\sm e$ is 3-connected after 
suppressing 
edges of degree 2; note that $G$ cannot be $K_4$ as the disconnected set $X$ contains at least 
three edges.
Let $G'$ be the graph obtained from $G\sm e$ by contracting all but one edge from every serial 
class.

By construction, any vertex of degree 2 of $G\sm e$ must be an endvertex of $e$. 
Hence every nontrivial serial class has size two and there are at most two of them. 
If a serial class contains an edge in $X$ and an edge not in $X$, we contract the edge of $X$ in 
the construction of $G'$.
This construction ensures that we never contract all edges of a path in $G$ that connects two 
components of $G[X]$. We let $X'=X\cap E(G')$. 
Hence the components of $G'[X']$ come from those components of $G[X]$ such that not all their edges 
got contracted.

Thus $G'[X']$ is disconnected unless $G[X]$ has precisely three components and two of these 
components just consist of a single edge.
Furthermore both endvertices of $e$ have degree 3 and each of them is incident with one of these 
components consisting of a single edge.
In this case we say that the edge $e$ \emph{H-shaped}.

Since the edge $e$ was arbitrary, we find the desired constraint graph $(G',X')$ unless every edge 
of $G$ not in $X$ is $H$-shaped.
Since $G$ is connected, every component of $G[X]$ is incident with an edge not in $X$. Hence $G[X]$ 
has precisely three components and they all consist of single edges. Furthermore every vertex of 
$G$ is incident with one edge in $X$ and two edges not in $X$. Thus $G$ has precisely six vertices. 
The edges not in $X$ form a vertex-disjoint union of cycles. So as $G$ is a simple graph, 
they either form two vertex-disjoint triangles or a 6-cycle. In the first case it is 
straightforward to check that $(G,X)$ is the weird prism. In the second case it is 
straightforward to check that $(G,X)$ is isomorphic to the constraint 
Wagner graph or the Wagner prism. 
\end{proof}

This completes the proof of \autoref{situation3}. Our next step is to prove the following.

\begin{lem}\label{not_X_essential}
 Let $G$ be a 3-connected graph and let $X$ be an edge set of $G$ such that $G[X]$ has precisely 
two components. % $C_1$ and $C_2$.
Let $e\in E(G)\sm X$ be essential. Then one of the following holds.
\begin{enumerate}
 \item $e$ joins the two components of $G[X]$; or
\item there is a component $C$ of $G[X]$ that consists only of a single edge and $e$ has an 
endvertex $v$ of degree three that is incident with that edge and the third edge incident with $v$ 
joins the two components of $G[X]$; or
\item there is a component $C$ of $G[X]$ that consists of precisely two edges, which form a 
triangle together with $e$. The two endvertices of $e$ have degree 3 and are each incident with an 
edge 
that joins the two components of $G[X]$. 
\end{enumerate}
\end{lem}

\begin{proof}
We assume that $e$ does not join the two components of $G[X]$, in particular $G$ is not $K_4$.
If the simple graph $G'=G/e$ is 3-connected, then $(G', X\cap E(G'))$ is a 3-connected constraint 
minor such that $X\cap E(G')$ is disconnected. 
Since $e$ is essential this is impossible. 
Hence by Bixby's Lemma  \cite{oxley2} the graph $G\sm e$ is 3-connected after 
suppressing edges 
of degree 2. 
Let $G'$ be the graph obtained from $G\sm e$ by contracting all but one edge from every serial 
class.

By construction, any vertex of degree 2 of $G\sm e$ must be an endvertex of $e$. 
Hence every nontrivial serial class has size two and there are at most two of them. 
If a serial class contains an edge in $X$ and an edge not in $X$, we contract the edge of $X$ in 
the construction of $G'$.
This construction ensures that we never contract all edges of a path in $G$ that connects two 
components of $G[X]$.
We let $X'=X\cap E(G')$. 
Hence the components of $G'[X']$ come from those components of $G[X]$ such that not all their edges 
got contracted.
Since $G'$ is 3-connected and $e$ is essential, the graph $G'[X']$ is connected.

Hence there must be a component $C$ of $G[X]$ such that all its edges got contracted. Hence 
$C$ has 
at most two edges. 
We split into two cases.

\paragraph{Case 1: $C$ has only a single edge $f$.}

Then $e$ has an endvertex $v$ of degree 3 that is incident with $f$. 
In this case we shall show that we have outcome 2; that is, the third edge $g$ incident with $v$ 
joins the two 
components of $G[X]$. Indeed, we construct $G''$ like $G'$ but instead of $f$ we contract $g$. 
Since $G''$ is isomorphic to $G'$, it is 3-connected. As $e$ is essential, it must be that 
$G''[X'+f]$ is connected.
Since the component of $G[X]$ different from $C$ does not contain a vertex incident with $e$, the 
edge $g$ joins the two components of $G[X]$. 

\paragraph{Case 2: $C$ has two edges $f_1$ and $f_2$.}

Then $e$ has two endvertices $v_1$ and $v_2$ of degree three such that $v_i$ is incident with 
$f_i$. Since $G$ is a simple graph and $C$ is connected,
the three edges $e$, $f_1$ and $f_2$ form a triangle. Similar as in Case 1 we prove for each $i$ 
that the third 
edge incident with $v_i$ joins the two components of $G[X]$. So we have outcome 3 in this case. 
\end{proof}

The following lemma deals with outcome 2 of \autoref{not_X_essential}. 

\begin{lem}\label{situation2}
Let $G$ be a 3-connected graph and $X$ a disconnected edge set of $G$.
Assume that every edge not in $X$ is essential.
Assume that a component $C$ of $G[X]$ consists only of a single edge and that there is an edge $vw$ 
such that $v$ is a vertex of $C$ and $w$ is not in $G[X]$. 
Then $(G,X)$ is the constraint wheel. 
\end{lem}

\begin{proof}
 The constraint graph $(G,X)$ is not the weird prism; indeed the weird prisms has no edge $vw$ as 
required in the assumptions. Hence by \autoref{situation3}, 
$G[X]$ has only one connected component $C'$ aside from $C$.
The endvertex $w$ of $e$ that is not in $C$ is not incident with any edge of $X$. 
Since $G$ is 3-connected, $w$ is incident with at least two edges $f_1$ and $f_2$ aside from $e$. 
By \autoref{not_X_essential} the endvertex of each $f_i$ different from $w$ must be in $C$ or $C'$. 
Since $C$ has only one vertex aside from $v$, one of the $f_i$ must have an endvertex in $C'$. By 
symmetry, we may assume that this is true for $f_1$. 
Since $f_1$ has an endvertex that is in neither $C$ nor $C'$, we can apply 
\autoref{not_X_essential} to deduce that $C'$ also consists of a single edge. 

\begin{sublem}\label{xyz}
 The vertex set of $G$ is $(C\cup C')+w$.
\end{sublem}

\begin{proof}
By \autoref{not_X_essential}, each vertex of $C\cup C'$ that has a neighbour outside that set has 
degree three and at most one neighbour outside that set. 
Let $W$ be the set of vertices of  $C\cup C'$ that have a neighbour outside the set $(C\cup 
C')+w$. Since $w$ has at least three neighbours in 
$C\cup C'$, the set $W$ contains at most one vertex.
The set $W$ together with $w$ separates $G$ if there are vertices not in $(C\cup C')+w$.
Since $G$ is 3-connected, this is not true. Hence $(C\cup C')+w$ is the vertex set of $G$.
\end{proof}

Since $w$ is adjacent to at least three vertices in $C\cup C'$, at least three vertices of $C\cup 
C'$ have precisely two neighbours in $C\cup C'$. Hence the graph
$G[C\cup C']$ is a 4-cycle. Since $G$ is 3-connected, each of its vertices has degree at least 
three. Hence by 3-connectivity every vertex of $C\cup C'$ is adjacent to $w$.
Thus $G$ is the constraint wheel.
\end{proof}

Given an edge set $Z$, by $V(Z)$ we denote the set of endvertices of edges in $Z$. 
Summing up, we have the following.

\begin{lem}\label{summary}
 Let $(G,X)$ be a 3-connected constraint graph such that $X$ is disconnected.
 Assume that every edge not in $X$ is essential and that $(G,X)$ is neither the constraint wheel 
nor the weird prism.
 Then $G[X]$ has precisely two connected components $C_1$ and $C_2$. All edges not in $X$ have 
both their endvertices in $V(X)$. 
\end{lem}

\begin{proof}By assumption and by \autoref{situation3}, $G[X]$ has precisely two 
connected components,  $C_1$ and $C_2$. By \autoref{not_X_essential} and \autoref{situation2}, 
every edge not in $X$ has both its endvertices in $V(X)$. 
 
\end{proof}

\section{Contracting edges in the constraint}

In this section we prove \autoref{main_intro}. 

First we need some preparation. 
Given a bond $d$ in a graph $G$, then $G-d$ has two connected components which we 
call the 
\emph{sides of} $d$. If we want to specify them, we call them the \emph{left side} and the 
\emph{right side}.

Given a graph $G$ and a bond $d$ of $G$, we say that $G$ is \emph{3-connected along $d$} if 
$G$ is 2-connected and there does not exist a separator consisting of two vertices from either  
side of $d$.

For the rest of this section we fix a graph $Q$ and a bond $d$ of $Q$ so that $Q$ is 3-connected 
along $d$.
We denote the set of edges on the left side of $d$ by $L$, and the set of edges on the right side 
of 
$d$ by $R$.
We assume throughout that $L$ and $R$ are nonempty. 
A \emph{special contraction minor} of $(Q,d)$ is a pair $(Q',d')$, where $Q'$ is obtained from $Q$ 
by contracting edges not in $d$, and $d'=d\cap E(Q')$.
Note that $d'$ and $d$ need not be equal as contractions might force us to delete edges in 
parallel classes. Since any parallel class 
containing one edge of $d$ is a subset of $d$, the set $d'$ is independent of the choice of the 
deleted edges.

\begin{eg}The following pairs $(Q,d)$ will be of particular interest in this paper.
For any two bonds of $K_4$ with both sides nonempty, there is an isomorphism of $K_4$ that induces 
a bijection between these 
two bonds.
The \emph{special $K_4$} is the pair consisting of the graph $K_4$ and a bond of size 4.
The \emph{special prism} is the pair consisting of the prism  and a bond whose complement consists 
of the two triangles of the prism, see \autoref{fig:special}.

\begin{figure}[htbp]
	\centering
	\subfloat[the special $K_4$\label{fig:special_k4}]{\includegraphics{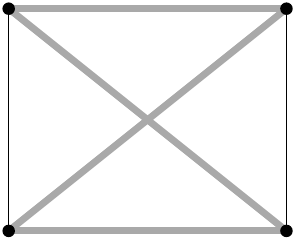}}
	\hspace{1cm}
	\subfloat[the special prism \label{fig:special_prism}]{\includegraphics{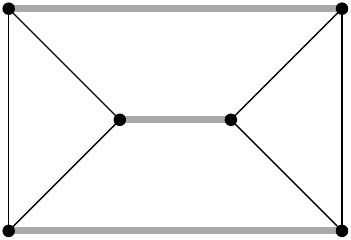}}
	\caption{The edges in the bond $d$ are coloured grey.}
	\label{fig:special}
\end{figure}
\end{eg}

Our aim in this section is to prove the following.

\begin{lem}\label{minor_char_H}
 Let $Q$ be a graph 3-connected along a bond $d$ such that the two sides of $d$ contain edges.
Then $(Q,d)$ has a special contraction minor that is the special $K_4$ or the special prism.
\end{lem}

\begin{proof}[Proof that \autoref{minor_char_H} implies \autoref{main_intro}.]
Let $(G,X)$ be a 3-connected constraint graph such that $X$ is disconnected. Our aim is to show 
that $(G,X)$ has the constraint $K_4$, the constraint wheel or a constraint prism as a 
constraint minor. By picking $(G,X)$ minimal, we may assume that every edge not in $X$ is 
essential. By \autoref{bad_prims} we may assume that $(G,X)$ is not the weird prism. We may also 
assume that it is not the constraint wheel. Thus by \autoref{summary}, $G[X]$ has precisely two 
connected components $C_1$ and $C_2$. And all edges not in $X$ have 
both their endvertices in $V(X)$. We take $Q=G$ and $d$ to be the bond consisting of those edges 
with one endvertex in $C_1$ and the other in $C_2$. Note that each $C_i$ contains at least one 
edge. Since $Q$ is 3-connected, $(Q,d)$ is 3-connected along $d$. 

By \autoref{minor_char_H}, $(Q,d)$ has a special contraction minor $(Q',d')$ that is the special 
$K_4$ or the special prism. Put another way, we can contract edges not in $d$ such that $G$ 
is $K_4$ or 
the prism. Let $X'=X\cap E(Q')$. 
We recall that if contractions force us to delete edges from a parallel class 
we first delete edges not in $X$. Hence since $X$ spans the two sides of $d$ in $G$, also $X'$ 
spans the two sides of $d'$ in $Q'$. Thus if $(Q',d')$ is a special $K_4$, then $(G,X)$ has the 
constraint $K_4$ as a constraint minor. Otherwise  $(Q',d')$ is the special prism. It is 
straightforward to check that in this case $(G,X)$ has a 
constraint prism as a constraint minor.
\end{proof}

The rest of this section is dedicated to the proof of \autoref{minor_char_H}. 
A pair $(Q,d)$ is \emph{irreducible} if $Q$ is 3-connected along $d$ but 
there does not exist a proper\footnote{non-identical} special contraction minor $(Q',d')$ such that 
both sides of $d'$ contain edges and $Q'$ is 3-connected along $d'$. 
The first step in the proof of \autoref{minor_char_H} will be to show that the set of irreducible 
 $(Q,d)$ is bounded. Later we examine this bounded set. 

Given an edge set $Z$ of $Q$, by $Q[Z]$ we denote the subgraph of $Q$ whose vertices are those with 
at least one endvertex in $Z$ and whose edges are those in $Z$.

\begin{lem}\label{not_2_con}
If the graph $Q[L]$ is not 2-connected and has at least two edges, then $(Q,d)$ is not irreducible.
\end{lem}

\begin{proof}
We consider the block-cutvertex-tree of $Q[L]$ and take a leaf block $b$. Recall that $b$ is a 
2-connected subgraph of $Q[L]$ or a single edge attached at a cutvertex $v\in b$ to the rest of 
$Q[L]$. We obtain $Q_1$ from $Q$ by contracting all edges of $Q[L]$ not in $b$. Since by assumption 
there is an edge in $Q[L]$ that is not in $b$, $Q_1$ is a nontrivial contraction of $Q$. 

Next we consider the block-cutvertex-tree of $Q[R]$. Note that unlike 
that for $Q[L]$ this may consist of just a single node. We obtain $Q_2$ from $Q_1$ by successively 
contracting leaf blocks $b'$ attached at a cutvertex $v'$ onto $v'$ if there is no edge 
between $b-v$ and $b'-v'$.  

In a slight abuse of notation we denote the contraction vertex of $Q_2$ containing $v$ by $v$. 
Similarly after contracting a leaf part on the right side, we denote the contraction vertex 
containing $v'$ by $v'$. 
We let $d_2=d\cap E(Q_2)$. We denote the edges on the left of $d_2$ by $L_2$ and the edges on the 
right of $d_2$ by $R_2$.

Our aim is to show that $Q_2$ is 3-connected along $d_2$.
By construction $L_2$ is nonempty.

\begin{sublem}\label{is_nonempty}
The edge set $R_2$ is nonempty. 
\end{sublem}
\begin{proof}
In the construction of $Q$ we only contract a leaf block $b'$ on the right side attached with 
cutvertex $v'$ if there is no edge between $b-v$ and $b'-v'$. In particular by contraction we never 
identify two vertices of $Q[R]$ that have neighbours in $b-v$. 

If there was only a single vertex $z$ in $Q[R]$ that has a neighbour in $b-v$, then $Q-v-z$ would 
be disconnected, contrary to our assumption that $Q$ is 3-connected along $d$. Hence there are at 
least two vertices in $Q[R]$ that have neighbours in $b-v$. Thus as explained above, the connected 
graph $Q_2[R_2]$ contains at least two vertices. Hence $R_2$ contains an edge. 
\end{proof}

\begin{sublem}\label{is_2con}
The graph $Q_2$ is 2-connected. 
\end{sublem}

\begin{proof}
Let $x$ be an arbitrary vertex of $Q_2$. We distinguish two cases.

{\bf Case 1: } $x=v$. 

By \autoref{is_nonempty}, the connected graph $Q_2[R_2]$ has a neighbour in the connected set 
$b-v$. Hence $Q_2-x$ is connected. 

{\bf Case 2: } $x\neq v$. If $x$ is not a contraction vertex, then $Q_2-x$ is connected as $Q-x$ is 
connected. So $x$ is a vertex of $Q_2[R_2]$. Let $K$ be a component of the graph $Q_2[R_2]-x$. Let 
$K'$ be the component of $Q[R]-x$ containing $K$. Since $Q-x$ is connected, there is an edge from 
$K'$ to $Q[L]$. Hence there is an edge 
from $K$ to $b$ in $Q_2$. Hence every component of the graph $Q_2[R_2]-x$ sends an edge to the 
connected set $b$. Hence  $Q_2-x$ is connected.  
\end{proof}

\begin{sublem}\label{is_3con}
For any two vertices $x\in Q_2[L_2]$ and $y\in Q_2[R_2]$ the graph $Q_2-x-y$ is connected.  
\end{sublem}

\begin{proof}
We distinguish two cases.

{\bf Case 1: } $x=v$.

Let $K$ be a component of the graph $Q_2[R_2]-y$. Since $K$ did not get contracted, it has a 
neighbour in $b-v$. Thus every component of  $Q_2[R_2]-y$ has a neighbour in the connected set 
$b-v$. Hence $Q_2-x-y$ is connected. 

{\bf Case 2: } $x\neq v$. 

Let $K$ be a component of the graph $Q_2[R_2]-y$. Let $K'$ be the component of $Q[R]-y$ containing 
$K$. Since $Q-x-y$ is connected, there is an edge from $K'$ to $Q[L]-x$. Hence there is an edge 
from $K$ to $b-x$ in $Q_2$. Hence every component of the graph $Q_2[R_2]-y$ sends an edge to the 
connected set $b-x$. Hence  $Q_2-x-y$ is connected. 
\end{proof}

By \autoref{is_2con} and \autoref{is_3con}, $Q_2$ is 3-connected along $d_2$. By construction $Q_2$ 
is obtained from $Q$ by contracting at least one edge. By \autoref{is_nonempty}, the edge sets 
$L_2$ and $R_2$ are nonempty. Hence $(Q_2,d_2)$ witnesses that $(Q,d)$ is not irreducible.

\end{proof}

\begin{lem}\label{both_2_con}
If the graph $Q[L]$ is 2-connected but not a triangle and the graph $Q[R]$ is 2-connected or 
consists of a single edge, then $(Q,d)$ is not irreducible. 
\end{lem}

In the proof of \autoref{both_2_con} we shall use the following lemma.
An edge $e$ in a 2-connected graph $G$ is \emph{contractible} if $G/e$ is 2-connected.

\begin{lem}\label{tutte_lem}
 If $G$ is a 2-connected graph that is not a triangle, then it has four contractible edges, 
two of which do not share an endvertex.
\end{lem}

\begin{proof}
 If $G$ is 3-connected or a cycle of length at least 4, every edge is contractible and the lemma is 
true in this case. Hence the Tutte-decomposition \cite{TutteGrTh} of $G$ has at least two 
leaf parts. The 
torsos of these parts are cycles or 3-connected. 
Let $v$ be a vertex in a leaf part that is not in the separator. Then any edge incident with $v$ is 
contractible. Since there are at least two leaf parts, we can pick vertices $v$ in one of each. 
Each such vertex is incident with at least two edges and no edge is incident with both these 
vertices. So there are at least four contractible edges, and there are two of them that do 
not share an 
endvertex.
\end{proof}

\begin{proof}[Proof of \autoref{both_2_con}.] 
  Suppose for a contradiction that  $(Q,d)$ is irreducible. 
 Let $vw$ be a contractible edge of $Q[L]$ (which exists by \autoref{tutte_lem}).

\begin{sublem}\label{sublem7}
 $Q/vw$ is 2-connected.
\end{sublem}
\begin{proof}
 As $Q$ is 2-connected and $Q/vw$ is a contraction, it suffices to show that $Q-v-w$ is connected.
 Since $vw$ is a contractible edge of $Q[L]$, the set $Q[L]-v-w$ is connected. So either $Q-v-w$ is 
connected or else the connected set $Q[R]$ can only have $v$ or $w$ as neighbours in $Q[L]$. 

Hence we may assume that we have the second outcome. 
Our aim is to derive a contradiction in that case. More precisely, we show that $(Q,d)$ is not 
irreducible. 
We obtain $\hat Q$ from $Q$ by contracting a 
spanning tree of $Q[L]-v-w$ and an edge from that set to one of $v$ or $w$. Note that $\hat Q$ is 
isomorphic to the graph obtained from $Q$ by deleting $Q[L]-v-w$. In our notation we suppress this 
bijection and just say things like `$v$ and $w$ are vertices of $\hat Q$'. 

Our aim is to show that $\hat Q$ is 3-connected along $d$. Suppose not for a contradiction. Then 
there is a separating set $S$ witnessing that. Let $a$ and $b$ be two vertices in different 
components of $\hat Q-S$. Let $P$ be a path in $Q-S$ joining $a$ and $b$. If $P$ contains a vertex 
of $Q[L]-v-w$, we can shortcut it by the edge $vw$. Hence we may assume that $P$ contains no vertex 
of  $Q[L]-v-w$. So $P$ is a path in $\hat Q-S$. This is a contradiction to the assumption that $a$ 
and $b$ are separated by $S$. Hence $\hat Q$ is 3-connected along $d$. As both sides of $d$ in 
$\hat Q$ contain edges, $(\hat Q,d)$ witnesses that $(Q,d)$ is not irreducible. This is the desired 
contradiction. 
\end{proof}

We abbreviate $Q'=Q/vw$.
Let $d'=d\cap E(Q')$. Let $L'$ be the left side of $d'$. The right side of $d'$ is $R$. 

\begin{sublem}\label{sublem71}
If $Q'$ is not 3-connected along $d'$, there is a vertex $z$ of $Q[R]$ such that $Q[L]-v-w$ can 
only have $z$ as a neighbour in $Q[R]$. 
\end{sublem}

\begin{proof}
By \autoref{sublem7}, there are vertices $y$ of $Q'[L']$ and $z$ of $Q'[R]$ such that $Q'-y-z$ is 
disconnected. Since $Q$ is 3-connected along $d$ and $Q'$ is a contraction of $Q$, it must be that 
$y$ or $z$ is a contraction vertex. Hence $y$ is the vertex $vw$. Hence $Q-v-w-z$ is disconnected.
Since $vw$ is contractible, $Q[L]-v-w$ is connected. By assumption $Q[R]-z$ is connected. So  
$Q[L]-v-w$ 
has no neighbour in $Q[R]-z$. 
\end{proof}
 
  By \autoref{tutte_lem}, $Q[L]$ has three contractible edges $a_1a_2$, $b_1b_2$ and $c_1c_2$ such 
that $a_1$, 
$a_2$, $b_1$ 
and $b_2$ are distinct vertices. Applying  \autoref{sublem71} to $a_1a_2$ and $b_1b_2$ yields that 
there are 
at most two vertices of $Q[R]$ that have neighbours in $Q[L]$. There have to be two such vertices 
as $Q$ is 2-connected. Call these vertices $z_1$ and $z_2$. \autoref{sublem71} gives the further 
information that one of them, say $z_1$, can only be incident to $a_1$ or $a_2$ and $z_2$ can only 
be 
to $b_1$ or $b_2$. 
 Now we apply  \autoref{sublem71} to $c_1c_2$. Since $c_1c_2$ is distinct from $a_1a_2$ and 
$b_1b_2$, there have to 
be vertices on these edges not in $c_1c_2$. By symmetry, we may assume that $a_1$ and $b_1$ are not 
in 
$c_1c_2$. Applying  \autoref{sublem71} to $c_1c_2$ yields that there is a single $z_i$ such that 
$a_1$ and 
$b_1$ can only have $z_i$ as a neighbour in $Q[R]$. By symmetry, we may assume that $z_i$ is 
equal to $z_1$. 
Hence $z_2$ can only have the neighbour $b_2$ in $Q[L]$. Hence $Q-z_1-b_2$ is disconnected. This is 
a 
contradiction to the assumption that $Q$ is 3-connected along $d$. Thus $(Q,d)$ is not irreducible. 
\end{proof}

\begin{lem}\label{edge_case}
If both graphs $Q[L]$ and $Q[R]$ consist of a single edge, then $(Q,d)$ is the special $K_4$.
\end{lem}
\begin{proof}
  Since every vertex is in $L$ or $R$, the graph $Q$ has precisely four vertices. Since no two 
vertices from different sides of $d$ separate,
$Q$ must contain all four edges joining the endvertices of these edges. Hence $Q$ is a the special 
$K_4$. 
\end{proof}

\begin{lem}\label{triangle_case}
If both graphs $Q[L]$ and $Q[R]$ are triangles, then $(Q,d)$ is the special prism or 
has a (proper) special $K_4$ as a special contraction minor. 
\end{lem}

\begin{proof}
 If $Q$ has only three edges between $Q[L]$ and $Q[R]$, then as $Q$ is 3-connected along $d$, these 
edges must form a matching. So $(Q,d)$ is the special prism. 

Thus we may assume that $Q$ has at least four edges between $Q[L]$ and $Q[R]$. So $Q[L]$ and $Q[R]$ 
each contain a vertex that has at least two neighbours on the other side. Call these vertices 
$\ell$ and $r$. Since $\ell$ and $r$ do not separate, there is an edge $\ell'r'$ between $Q[L]$ and 
$Q[R]$ that is not incident with $\ell$ and $r$. By symmetry, we may assume that $\ell$ and $\ell'$ 
are in $Q[L]$, and $r$ and $r'$ are in $Q[R]$. As $r$ has two neighbours in $Q[L]$, we can contract 
a single edge of $Q[L]$ different from $\ell\ell'$ such that $r$ is adjacent to the two remaining 
vertices of $Q[L]$. Similarly, we contract an edge of $Q[R]$ different from $rr'$ such that the 
vertex of $\ell$ is adjacent to the two remaining 
vertices of $Q[R]$. The resulting contraction is a special 
$K_4$.
\end{proof}

\begin{lem}\label{case_ana}
If $Q[L]$ is a single edge and $Q[R]$ is a triangle, then $(Q,d)$ has a special $K_4$ as a (proper)
special 
contraction minor. 
\end{lem}

\begin{proof}
We denote the edge in $Q[L]$ by $vw$. Since $Q$ is 2-connected, each of $v$ and $w$ has a neighbour 
in $Q[R]$. If one of them has only a single neighbour in $Q[R]$, then that neighbour together with 
the other endvertex of $vw$ is 2-separator. This is impossible as $Q$ is 3-connected along $d$.

Hence $v$ and $w$ have each at least two neighbours in $Q[R]$. So there is a vertex $x$ in 
$Q[R]$ adjacent to $v$ and $w$. Contracting the edge not incident with $x$ to a single vertex, 
yields a special $K_4$ as a special contraction minor. 
\end{proof}

\begin{proof}[Proof of \autoref{minor_char_H}.]
By taking $(Q,d)$ contraction-minimal, we may assume that it is irreducible. We will show that 
$(Q,d)$ is a 
special prism or a special $K_4$.
If both graphs $Q[L]$ and $Q[R]$ are 2-connected, then by \autoref{both_2_con} (and the same lemma 
applied with the roles of `$L$' and `$R$' interchanged) both of them are 
triangles. In this case, by \autoref{triangle_case} $(Q,d)$ is a special prism. 

Otherwise one of $Q[L]$ or 
$Q[R]$ is not 2-connected. By \autoref{not_2_con} (and the same lemma 
applied with the roles of `$L$' and `$R$' interchanged) 
it consists of a single edge. Hence we may assume that one of the two graphs $Q[L]$ and $Q[R]$ must 
be a single edge. 
By combining \autoref{not_2_con} with \autoref{both_2_con}, we deduce that the other graph must be a 
single edge or a triangle.
It cannot be a triangle by \autoref{case_ana}. Hence $(Q,d)$ is a special $K_4$ by 
\autoref{edge_case} in this case. 
\end{proof}

  \begin{proof}[Proof of \autoref{main_intro}.]
  We have just finished the proof of \autoref{minor_char_H}. And just after the statement of that 
lemma we showed that it implies \autoref{main_intro}. 
  \end{proof}

\section{Concluding remarks}\label{sec:conj}

There are various ways how \autoref{main_intro} might be extended. First, can we replace 
`constraint connectedness' by the property that the set $X$ has at most $k$ connected components 
for some natural number $k$? More precisely, a constraint graph $(G,X)$ has at \emph{most $k$ 
islands} if $G[X]$ has at most $k$ connected components. Clearly, the class of constraint graph 
with at most $k$ islands is closed under taking constraint minors. 

\begin{con}\label{kcon}
 Let $k>1$. The class of 3-connected constraint graphs with at most $k$ islands is 
characterised by a finite list of excluded constraint minors.
\end{con}
Can you explicitly compute the list of excluded minors in \autoref{kcon}?

Another extension is as follows. A \emph{double-constraint matroid} $(M,X,Y)$ consists of a 
matroid $M$ and two sets $X$ and $Y$ of its elements. It is \emph{realisable} if $M$ is the cycle 
matroid of 
a graph $G$ such that both $X$ and $Y$ are connected in $G$. Can you extend \autoref{new_intro} 
from constraint matroids to double-constraint matroids? Put another way: is a double-constraint 
matroid realisable if and only if it does not have one of finitely many excluded double-constraint 
minors? Although for 3-connected matroids, the answer to this question follows from 
\autoref{new_intro}, for matroids that are not 3-connected new obstructions arise, see 
\autoref{fig:2con}
   \begin{figure} [htpb]   
\begin{center}
   	  \includegraphics[height=2cm]{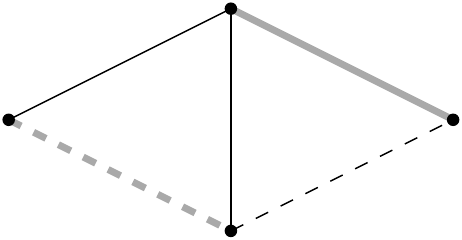}
   	  \caption{The constraint $X$ is depicted in grey, the constraint $Y$ is 
dashed. Although the matroid represented by this graph is realisable for each of 
$X$ or $Y$, it is not realisable for both of them at the same time.}\label{fig:2con}
\end{center}
   \end{figure}

\bibliographystyle{plain}
\bibliography{literatur}

\end{document}

%% file: minimalJ.tex
\usepackage{epsfig}
\usepackage{graphicx}
\usepackage{amsbsy}
\usepackage{amsmath}
\usepackage{amsfonts}
\usepackage{amssymb}
\usepackage{textcomp}
\usepackage{hyperref}
\usepackage{aliascnt}

\newcommand{\mcm}[3]{\newcommand{#1}[#2]{{\ensuremath{#3}}}} 

\mcm{\tuple}{1}{\langle #1 \rangle}
\mcm{\name}{1}{\ulcorner #1 \urcorner}
\mcm{\Nbb}{0}{\mathbb{N}}
\mcm{\Zbb}{0}{\mathbb{Z}}
\mcm{\Rbb}{0}{\mathbb{R}}
\mcm{\Cbb}{0}{\mathbb{C}}
\mcm{\Qbb}{0}{\mathbb{Q}}
\mcm{\Acal}{0}{\cal A}
\mcm{\Bcal}{0}{\cal B}
\mcm{\Ccal}{0}{\cal C}
\mcm{\Dcal}{0}{\cal D}
\mcm{\Ecal}{0}{\cal E}
\mcm{\Fcal}{0}{\cal F}
\mcm{\Gcal}{0}{\cal G}
\mcm{\Hcal}{0}{\cal H}
\mcm{\Ical}{0}{\cal I}
\mcm{\Jcal}{0}{\cal J}
\mcm{\Kcal}{0}{\cal K}
\mcm{\Lcal}{0}{\cal L}
\mcm{\Mcal}{0}{\cal M}
\mcm{\Ncal}{0}{\cal N}
\mcm{\Ocal}{0}{{\cal O}}
\mcm{\Pcal}{0}{{\cal P}}
\mcm{\Qcal}{0}{{\cal Q}}
\mcm{\Rcal}{0}{{\cal R}}
\mcm{\Scal}{0}{{\cal S}}
\mcm{\Tcal}{0}{{\cal T}}
\mcm{\Ucal}{0}{{\cal U}}
\mcm{\Vcal}{0}{{\cal V}}
\mcm{\Wcal}{0}{{\cal W}}
\mcm{\Xcal}{0}{{\cal X}}
\mcm{\Ycal}{0}{{\cal Y}}
\mcm{\Mfrak}{0}{\mathfrak M}

\mcm{\restric}{0}{\upharpoonright}
\mcm{\upset}{0}{\uparrow}
\mcm{\onto}{0}{\twoheadrightarrow}
\mcm{\smallNbb}{0}{{\small \mathbb{N}}}
\DeclareMathOperator{\preop}{op}
\mcm{\op}{0}{^{\preop}}

\newcommand{\se}{\subseteq}
{\begin{array}{c}
\setlength{\unitlength}{1em}}%
{\end{array}}

\usepackage{amsthm}

\newcommand{\theoremize}[2]{\newaliascnt{#1}{thm} \newtheorem{#1}[#1]{#2} \aliascntresetthe{#1}}

\theoremstyle{plain}
\newtheorem{thm}{Theorem}[section]
\theoremize{lem}{Lemma}
\theoremize{fact}{Fact}
\theoremize{sublem}{Sublemma}
\theoremize{claim}{Claim}
\theoremize{rem}{Remark}
\theoremize{obs}{Observation}
\theoremize{prop}{Proposition}
\theoremize{cor}{Corollary}
\theoremize{que}{Question}
\theoremize{oque}{Open Question}
\theoremize{con}{Conjecture}

\theoremstyle{definition}
\theoremize{dfn}{Definition}
\theoremize{eg}{Example}
\theoremize{exercise}{Exercise}
\theoremstyle{plain}

%% file: constraint_minors170824.bbl
\begin{thebibliography}{1}

\bibitem{3space4}
J.~Carmesin.
\newblock Embedding simply connected 2-complexes in 3-space {IV}: dual
  matroids.
\newblock Preprint 2017.

\bibitem{DiestelBookCurrent}
R.~Diestel.
\newblock {\em Graph {T}heory \emph{(5th edition)}}.
\newblock Springer-Verlag, 2016.
\newblock \\ Electronic edition available at:\\ {\small\tt
  http://diestel-graph-theory.com/index.html}.

\bibitem{oxley2}
J.~Oxley.
\newblock {\em {M}atroid {T}heory \emph{(2nd edition)}}.
\newblock Oxford University Press, 2011.

\bibitem{TutteBook}
W.~T. Tutte.
\newblock {L}ectures on matroids.
\newblock {\em J. Res. Nat. Bur. Standards Sect. B}, 69B:1--47, 1965.

\bibitem{TutteGrTh}
W.~T. Tutte.
\newblock {\em Graph theory}, volume~21 of {\em Encyclopedia of Mathematics and
  its Applications}.
\newblock Addison-Wesley Publishing Company, Advanced Book Program, Reading,
  MA, 1984.
\newblock With a foreword by C. St. J. A. Nash-Williams.

\bibitem{Whitney_flip}
H.~Whitney.
\newblock 2-{I}somorphic {G}raphs.
\newblock {\em Amer. J. Math.}, 55:245--254, 1933.

\end{thebibliography}
